\documentclass[12pt,a4paper,reqno]{amsart}
\usepackage{fullpage}
\usepackage{amsmath,amssymb}
\usepackage{newtxtext,newtxmath}
\usepackage{hyperref}
\usepackage{url}
\usepackage{xcolor}
\usepackage{tikz}
\usepackage{comment}

\usepackage{graphicx}
\graphicspath{{figures/}}

\newcommand{\Odip}[2]{O_{#1}\left(#2\right)}
\newcommand{\Odi}[1]{O\left(#1\right)}
\newcommand{\dx}{\,\mathrm{d}}
\newcommand{\e}{\mathrm{e}}
\newcommand{\PP}{\mathfrak{P}}

\newcommand{\E}{\mathcal{E}}
\newcommand{\N}{\mathbb{N}}
\newcommand{\R}{\mathbb{R}}
\newcommand{\Z}{\mathbb{Z}}

\allowdisplaybreaks

\newcommand{\A}{\mathcal{A}}
\newcommand{\abs}[1]{\left\vert#1\right\vert}

\newcommand{\fr}[2]{\left(\frac{#1}{#2}\right)}
\newcommand{\LL}{\mathcal{L}}
\newcommand{\U}{\mathcal{U}}

\DeclareMathOperator\li{li}

\newtheorem{Theorem}{Theorem}[section]

\newtheorem{Lemma}{Lemma}[section]

\allowdisplaybreaks

\author{Alessandro Gambini, Remis Tonon, and Alessandro Zaccagnini}

\title{On the distribution of the digits of quotients \\ of integers and primes}

\date{\today}

\begin{document}

\begin{abstract}
We investigate the distribution of the digits of quotients of randomly chosen positive
integers taken from the interval $[1,T]$, improving the previously known error term for the counting
function as $T\to+\infty$. We also resolve some natural variants of the problem concerning points with prime coordinates and points that are visible from the origin.
\end{abstract}

\maketitle

\small
\keywords{\emph{Keywords}: Quotient of Random Integers; Lattice points; Elementary probability; Euler $\psi$ function}

\subjclass{\emph{2010 Mathematics Subject Classification}: Primary: 11P21; Secondary: 33B15, 11K99}

\normalsize
\section{General introduction}

The first attempts to study number-theoretical problems by means of probabilistic methods date back to the 19th century and involve mainly two mathematicians: first Gauss, who was interested in the number of products of exactly $k$ distinct primes below a certain threshold (and the solution of this problem for $k=1$ is the famous prime number theorem, proved by Hadamard and de la Vall\'ee Poussin, building on the ideas of Riemann, only in 1896; see \cite{davenport} chapter 18) and then Ces\`{a}ro, who showed in 1881 that the probability that two randomly chosen integers are coprime is $6/\pi^2$. In 1885 he then published a book \cite{Cesaro1885} collecting some interesting problems from some articles he published in \textit{Annali di matematica pura ed applicata}: the article we are interested in is \textit{Eventualit\'{e}s de la division arithm\'{e}tique}, which originally appeared as \cite{Cesaro1885b}. There he states that, dividing two random integers, the probability that the $i$-th digit after the decimal point is $r$ is given by
\begin{equation}
\label{base-10}
\frac{1}{20}+\frac{10^i}{2}\int_0^1\frac{1-\varphi}{1-\varphi^{10}}\varphi^{10^i-1+r}\,\dx\varphi.
\end{equation}
As a consequence, we have the somewhat surprising discovery that if
one takes two ``random'' positive integers $n$ and $m$ and considers
the distribution of the first decimal digit of their ratio $n /
m$, it is slightly more likely that this turns out to be $0$ rather
than, say, $1$.

Taking as a starting point this result and the Benford law
\cite{Benford1938}, which has a similar behaviour with regard to the
frequency distribution of the main digit in many real-life numerical
data sets, Gambini, Mingari Scarpello and Ritelli
\cite{GambiniMSR2012} studied the problem in more detail, and
recognised that the integral can be expressed in terms of the digamma
function
\[
\psi(x) := \frac{\Gamma'(x)}{\Gamma(x)}.
\]
The representation
\begin{equation*}
  \psi(z)
  =
  -
  \gamma
  -
  \frac1z
  +
  \sum_{k \ge 1} \frac z{k (z + k)}
  =
  -
  \gamma
  +
  \sum_{k \ge 0} \Bigl( \frac1{k + 1} - \frac1{k + z} \Bigr),
\end{equation*}
which is (5.7.6) of \cite{olver2010LBC}, led the authors to a
different form for the integral in \eqref{base-10}, which made them
suspect that an elementary proof of the result was possible. Indeed,
they were able to find it and this is the starting point for our
study.

In order to be more precise, we start giving some definitions.
We consider a number basis $b \ge 2$ and the corresponding set of digits
$S_b = \{ 0$, \dots, $b - 1 \}$.
Given a positive real number $x$ and a positive integer $i$, we are concerned with the $i$-th digit to the right of the point of the representation in base $b$ of $x$: we will call it $\phi(x; b; i)$.
We remark that $\phi(x; b; i)$ can be computed by means of
\[
  \phi(x; b; i)
  =
  \lfloor b \{ b^{i - 1} x \} \rfloor
  =
  \bigl\lfloor b^i x - b \lfloor b^{i - 1} x \rfloor \bigr\rfloor.
\]
In fact, this formula is correct for any $i \in \Z$, with the obvious
interpretation if $i < 0$.
We recall that $\lfloor x \rfloor \in \Z$ and $\{ x \} \in [0, 1)$
denote the integer and the fractional part of the real number $x$,
respectively, so that $x = \lfloor x \rfloor + \{ x \}$.

With $b$ and $i$ as above and a digit $r \in S_b$, we also define
$\Phi(T; b, r; i) := \vert \A(T; b, r; i) \vert$, where
\[
  \A(T; b, r; i)
  :=
  \{ (n, m) \in \N^2 \cap [1, T]^2 \colon \phi(n / m; b; i) = r \}.
\]
Throughout the paper, for brevity we often drop the dependency on $b$, $r$
and $i$ of our functions, whenever there is no possibility of
misunderstanding.
We recall that Gambini, Mingari Scarpello and Ritelli
\cite{GambiniMSR2012} implicitly obtained the asymptotic formula
\[
  \Phi(T; b, r; i)
  =
  c(b, r; i) T^2 + \Odi{T^{3 / 2}},
\]
as $T \to +\infty$, where $b$, $r$ and $i$ are fixed. The Authors considered couples of real numbers, both taken from
$[1,T]$, whereas we are only interested in points with integral coordinates.
The constant $c(b, r; i)$ is defined as an infinite series and can be
expressed by means of the digamma function, as follows:
\begin{equation}
\label{eq:constantdigits}
  c(b, r; i)
  =
  \frac1{2 b}
  +
  \frac12 b^i
  \int_0^1
    \frac{1 - \varphi}{1 - \varphi^b} \, \varphi^{b^i + r - 1} \, \dx \varphi
  =
  \frac1{2 b}
  +
  \frac12 b^{i - 1}
  \Bigl(
    \psi \Bigl( \frac{b^i + r + 1}b \Bigr)
    -
    \psi \Bigl( \frac{b^i + r}b \Bigr)
  \Bigr).
\end{equation}
This agrees with \eqref{base-10} by (5.9.16) of \cite{olver2010LBC}.
We remark here that this problem is appropriately situated among the
classical problems of counting lattice points that belong to some
region of the plane.
One of the most famous of these results is Minkowski's theorem, which states that every closed convex set in $\R^n$ that is symmetric with respect to the origin and with volume greater than $2^n$ contains at least one point with integer coordinates distinct from the origin. This theorem, proved in 1889, gave birth to a new branch of number theory: the geometry of numbers. Another famous result in this direction is Pick's theorem: proved in 1899 (see \cite{Pick1899}), it relates the area of a simple polygon with integer coordinates with the number of lattice points in its interior and the number of lattice points on its boundary. Finally, it is mandatory to refer to two of the most famous problems in analytic number theory, which are related to ours: Gauss' circle problem and Dirichlet's divisor problem, see \cite{hardy-wright} chapter 18. Both of them deal with counting points with integer coordinates belonging to some region delimited by conics: a circumference and a hyperbola. The two mathematicians were able to provide a formula with the area of the figure as a main term plus some error due to the integer points near the boundary. For example, Gauss proved that in the circle of radius $r$ there are
\[
\pi r^2 + E(r)
\]
points with integer coordinates, where $\abs{E(r)}\le 2\sqrt{2}\pi r$.
By this simple formulation, one could think that guessing the right order of magnitude of the error term should not be a hard problem; actually, although many (slow) improvements have been done, both problems remain open.

The interest in the random integer quotients also embraces other fields in analytic number theory. Recently there have been several papers dealing with the cardinalities of $A/A$ for subsets $A$ of the set of the first $n$ positive integers getting general lower and upper bounds, see Cilleruelo and Gujarro-Ord\'o\~nez \cite{cilleruelo2017G}, Cilleruelo et al \cite{cilleruelo2017RR, cilleruelo2010RR}.
Another line of research has to do with the search for prime numbers with a positive proportion of preassigned digits in base $b$ and the estimation of the number of these prime numbers, see Swaenepoel \cite{swaenepoel2020}.

\section{Results}

After this excursus, which gives some motivation to our research for a
better error term, we come back to our results. In this paper, we
introduce some number-theoretic devices which allow us to improve upon
the result by Gambini, Mingari Scarpello and Ritelli, and specifically
to obtain a better error term.
In all statements, we consider $b$, $r$ and $i$ fixed, so that, here
and throughout the paper, implicit constants may depend on them.
We also recall that $c(b, r; i)$ is the constant defined in
\eqref{eq:constantdigits}.

\begin{Theorem}
\label{Main-theorem}
As $T \to +\infty$ we have
\[
  \Phi(T; b, r; i)
  =
  c(b, r; i) T^2 + \Odi{T \log(T)}.
\]
\end{Theorem}

Our improvement stems largely from the fact that we evaluate more
carefully the error terms arising from computing ratios of integers
with the desired digit and that we introduce a variable threshold, to
be chosen at the end of the proof, which allows us to ignore some
points in $\A(T; b, r; i)$.

In the second part of the paper we deal with a variation of the same
problem: we consider the case of primes.
We let $\PP$ denote the set of positive prime integers.
For the sake of clarity, for $X \ge 2$ we let
\begin{equation}
\label{def-E}
  \E(X)
  :=
  \sup_{2 \le x \le X} \bigl\vert \pi(x) - \li(x) \bigr\vert,
\end{equation}
so that $\E$ is non-negative and increasing, and can be bounded by
means of the Prime Number Theorem, which is Lemma~\ref{Lemma-PNT}
below.

\begin{Theorem}
\label{th:theta}
In the case of primes, as $T \to +\infty$ we have
\[
  \sum_{(p, q) \in \A(T; b, r; i) \cap \PP^2} \log(p) \log(q)
  =
  c(b, r; i) T^2 + \Odi{ T \E(T) \log(T)}.
\]
\end{Theorem}

It is also possible to study the corresponding problem where only the
numerator or the denominator is restricted to being a prime, and the result is
similar.
See \S\ref{sec:primes} for some comments and details on the error term
in Theorem~\ref{th:theta}.

We also tackle another variant of
this problem and we consider only points that are visible from the
origin, casting out multiplicities.
We obtain our last result as a Corollary of Theorem~\ref{Main-theorem}, 
via M\"obius inversion.

\begin{Theorem}
\label{th:farey}
As $T \to +\infty$ we have
\[
  \sum_{\substack{(n, m) \in \A(T; b, r; i) \\ (n, m) = 1}} 1
  =
  \frac{c(b, r; i)}{\zeta(2)} T^2 + \Odi{T \log^2(T)},
\]
where $\zeta$ denotes the Riemann $\zeta$-function.
\end{Theorem}

\smallskip
\noindent{\textbf{Acknowledgements.}}
We thank Sandro Bettin for many conversations on the subject, and
for his help with the plots at the end of the present paper.

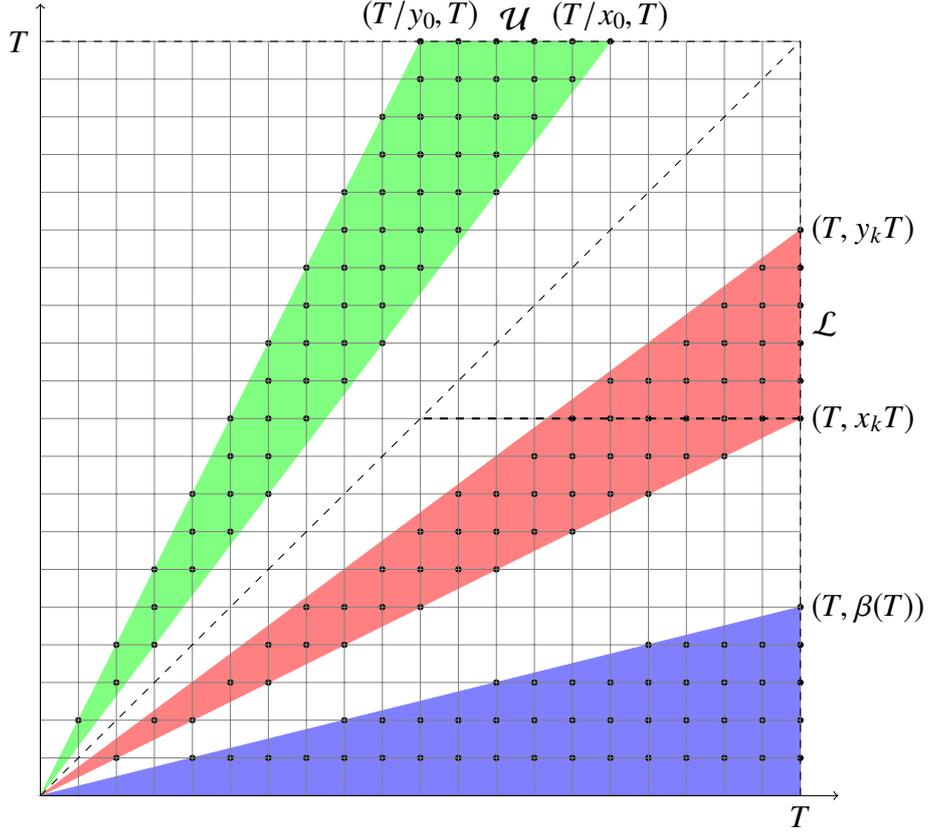
\begin{figure}
\[
\begin{tikzpicture}[scale=0.5]
  \pgfmathsetmacro{\xmax}{20};
  \pgfmathsetmacro{\ymax}{20};
  \pgfmathsetmacro{\delta}{1};
  \coordinate (A) at (\xmax,0);
  \coordinate (B) at (\xmax,\ymax);
  \coordinate (C) at (0,\ymax);

  \pgfmathsetmacro{\yuno}{5};
  \coordinate (uno) at (\xmax,\yuno);
  \filldraw[blue!50] (0,0) -- (uno) -- (\xmax,0) -- cycle;
  \draw (uno) node[right] {$(T,\beta(T))$};
  \foreach \i in {4, ..., \xmax} {
    \pgfmathsetmacro{\k}{\i*\yuno/\xmax};
    \foreach \j in {1, ..., \k}
      \filldraw (\i,\j) circle (2pt);
  }

  \pgfmathsetmacro{\xdue}{10};
  \pgfmathsetmacro{\xtre}{15};
  \coordinate[label=above:${(T / y_0, T)}$] (due) at (\xdue,\ymax);
  \coordinate[label=above:${(T / x_0, T)}$] (tre) at (\xtre,\ymax);
  \coordinate (n) at (barycentric cs:due=1,tre=1);
  \draw (n)[above] node {$\mathcal{U}$};
  \filldraw[green!50] (0,0) -- (due) -- (tre) -- cycle;
  \foreach \i in {1, ..., \xdue} {
    \pgfmathsetmacro{\h}{ceil(\i*\ymax/\xtre)};
    \pgfmathsetmacro{\k}{floor(\i*\ymax/\xdue)};
    \foreach \j in {\h, ..., \k}
      \filldraw (\i,\j) circle (2pt);
  }
  \foreach \i in {\xdue, ..., \xtre} {
    \pgfmathsetmacro{\h}{ceil(\i*\ymax/\xtre)};
    \foreach \j in {\h, ..., \ymax}
      \filldraw (\i,\j) circle (2pt);
  }

  \pgfmathsetmacro{\yquattro}{15};
  \pgfmathsetmacro{\ycinque}{10};
  \coordinate[label=right:${(T, y_k T)}$] (quattro) at (\xmax,\yquattro);
  \coordinate[label=right:${(T, x_k T)}$] (cinque)  at (\xmax,\ycinque);
  \coordinate[label=right:$\mathcal{L}$]
    (n) at (barycentric cs:quattro=1,cinque=1);

  \filldraw[red!50] (0,0) -- (quattro) -- (cinque) -- cycle;
  \foreach \i in {2, ..., \xmax} {
    \pgfmathsetmacro{\h}{ceil(\i/\xmax*\ycinque)};
    \pgfmathsetmacro{\hh}{\h + 1};
    \pgfmathsetmacro{\k}{floor(\i/\xmax*\yquattro)};
    \foreach \j in {\h, ..., \k}
      \filldraw (\i,\j) circle (2pt);
  }

  \draw[very thin,color=gray] (0,0) grid (\xmax,\ymax);
  \draw[dashed] (0,0) -- (\xmax,\ymax);
  \draw[dashed] (\xmax,0) -- (\xmax,\ymax) -- (0,\ymax);
  \draw[->] (0,0) -- (\xmax+\delta,0);
  \draw[->] (0,0) -- (0,\ymax+\delta);
  \draw (\xmax,0) node[anchor=north] {$T$};
  \draw (0,\ymax) node[anchor=east]  {$T$};

  \draw[dashed,thick] (cinque) -- (\ycinque,\ycinque);

\end{tikzpicture}
\]
\caption{\label{fig:suddivisione}
How to split the sets: here we illustrate the case $i = 1$.
The set $\U$ is a triangle.
The set $\LL$ is an infinite union of triangles; we estimate trivially
the contribution from triangles in the shaded region at the bottom.
In the paper, we consistently use $n$ for the abscissa and $m$ for the
ordinate of the points.}
\end{figure}

\section{Basic strategy of the proofs}
\label{sec:basic}

We follow \cite{GambiniMSR2012} quite closely.
The proofs share some common features and it is probably clearer if we
deal with them at the outset.
We remark that we may assume that $T$ is an integer, because the total
error involved in changing $T$ by a bounded amount is small: see the end
of this section.
We first decompose the set $\A(T; b, r; i)$ as an appropriate union of
sets.
For $r \in S_b$ we define
\begin{align*}
  y_k
  =
  y_k(b, r; i)
  &:=
  \frac{b^i}{b k + r} = \frac{b^{i - 1}}{k+\frac{r}{b}}, \\
  x_k
  =
  y_k(b, r+1;i)
  &:=
  \frac{b^i}{b k + r+1} =\frac{b^{i - 1}}{k+\frac{r+1}{b}}.
\end{align*}
With these definitions and Figure~\ref{fig:suddivisione} in mind,
we write
\begin{align}
\notag
  \A_k(T)
  =
  \A_k(T; b, r; i)
  &=
  \Bigl\{ (n, m) \in \N^2 \cap [1, T]^2 \colon
    \frac nm \in \Bigl[ \frac{k b + r}{b^i}, \frac{k b + r + 1}{b^i}\Bigr)
  \Bigr\} \\
\label{def-Ak-alt}
  &=
  \bigl\{ (n, m) \in \N^2 \cap [1, T]^2 \colon
    m \in \bigl( n x_k, n y_k \bigr]
  \bigr\} \\
\label{def-Ak-U}
  &=
  \bigl\{ (n, m) \in \N^2 \cap [1, T]^2 \colon
    n \in \bigl[ m / y_k, m / x_k \bigr)
  \bigr\},
\end{align}
so that $\A(T; b, r; i) = \bigcup_{k \ge 0} \A_k(T; b, r; i)$.
This is easily checked using the definition of $\phi$.
The sets $\A_k(T)$ are pairwise disjoint and correspond to the lattice
points contained in triangles with a vertex at the origin and the
other vertices either on the segment $[1, T] \times \{ T \}$, when
$0 \le k < b^{i - 1}$ (i.e. $n<m$), or on $\{ T \} \times [1, T]$,
when $k \ge b^{i - 1}$ (i.e. $n\ge m$), provided that $k$
satisfies \eqref{k-range} below.
Therefore, we split the infinite union above accordingly as
\[
  \U(T; b, r; i)
  :=
  \bigcup_{k = 0}^{b^{i - 1} - 1} \A_k(T; b, r; i)
  \qquad\text{and}\qquad
  \LL(T; b, r; i)
  :=
  \bigcup_{k \ge b^{i - 1}} \A_k(T; b, r; i).
\]
At this point, it is worth looking back at the definition of $\A_k(T)$.
When $k$ is such that
\[
  \frac1{y_k} = \frac{k b + r}{b^i} > T,
\]
the set $\A_k(T)$ is empty.
This means that we can bound the range for $k$, taking
\begin{align}\label{k-range}
  k\le b^{i-1}T - \frac rb \le b^{i-1}T.
\end{align}
Hence we have that
\[
  \LL(T; b, r; i)
  =
  \bigcup_{b^{i - 1} \le k \le b^{i-1}T} \A_k(T; b, r; i).
\]
For $k$ above a certain threshold depending on $T$, it is difficult
to evaluate the cardinality of $\A_k(T)$ exactly: this will be the source
of our first error term.

We notice here, even though we will need this later, that for $k\ge 1$
and $r \in S_b \cup \{ b \}$ we have
\[
\frac{1}{3k}
\le \frac{1}{k}\cdot\frac{1}{2+1/b}
= \frac{1}{k}\cdot\frac{1}{1+\frac{b+1}{b}}
< \frac{1}{k}\cdot\frac{1}{1+\frac{r+1}{bk}}
= \frac{1}{k+\frac{r+1}{b}}
<
\frac{1}{k+\frac{r}{b}}
\le\frac{1}{k},
\]
so that
\begin{equation*}
\begin{cases}
x_k\asymp {b^{i-1}}/k \\
y_k\asymp {b^{i-1}}/k.
\end{cases}
\end{equation*}

\subsection{Weights}

In order to treat our problems in a unified fashion, we introduce
weights associated to points with integral coordinates in $[1, T]^2$.
We will eventually choose the following weights: $\omega(n, m) := 1$
for all $n$ and $m$ in the problem with all integers;
\[
  \omega(n, m)
  :=
  \begin{cases}
    \log(n) \log(m) &\text{if $n$ and $m$ are both prime numbers,} \\
    0               &\text{otherwise}
  \end{cases}
\]
in the problem with primes; and
\[
  \omega(n, m)
  :=
  \begin{cases}
    1 &\text{if $(n, m) = 1$,} \\
    0 &\text{otherwise}
  \end{cases}
\]
in the problems with ``reduced'' couples.
We have to evaluate
\[
  \sum_{(n, m) \in \U(T; b, r; i)} \omega(n, m)
  \qquad\text{and}\qquad
  \sum_{(n, m) \in \LL(T; b, r; i)} \omega(n, m).
\]

With this choice of weights, we see that the error involved in
changing $T$ to the nearest integer is $\Odi{T}$ in the case of
integers and $\Odi{T \log(T)}$ in the case of primes.

\subsection{The contribution from the set
  \texorpdfstring{$\U(T; b, r; i)$}{U}}

We refer to Figure~\ref{fig:suddivisione}, and remark that
\begin{align}
\notag
  \sum_{(n, m) \in \U(T; b, r, i)} \omega(n, m)
  &=
  \sum_{k = 0}^{b^{i - 1} - 1}
    \sum_{(n, m) \in \A_k(T; b, r; i)} \omega(n, m) \\
\label{U-general}
  &=
  \sum_{k = 0}^{b^{i - 1} - 1}
    \sum_{m \in [1, T]}
      \sum_{n \in [m / y_k, m / x_k)}
        \omega(n, m).
\end{align}

\subsection{The contribution from the set
  \texorpdfstring{$\LL(T; b, r; i)$}{L}}

To a first approximation, $\vert \A_k(T) \vert$ is the area of the
corresponding triangle, with an error proportional to the perimeter,
that is $\Odi{T}$.
However, as $T\to+\infty$, the number of such triangles tends to
infinity as well,  and their area may be small, and we have to be much more careful than this.
We choose an appropriate function $\beta = \beta(T)$ and estimate
trivially the contribution from $k$ satisfying \eqref{k-range} with
$k > T / \beta(T)$.
It is
\begin{equation}
\label{contr-lower-triangle}
  \ll
  \sum_{n \le T}
    \sum_{m \le n \beta(T) / T}
      \omega(n, m)
  \le
  T \beta(T)
  \max_{n \le T} \max_{m \le n \beta(T) / T} \omega(n, m).
\end{equation}

Recalling \eqref{def-Ak-alt}, we see that the contribution from
$\A_k(T; b, r; i)$ is
\begin{align}
  &\sum_{m \le x_k T}
     \sum_{n \in [m / y_k, m / x_k)} \omega(n, m)
  +
  \sum_{x_k T < m \le y_k T}
     \sum_{n \in [m / y_k, T]} \omega(n, m) \nonumber \\
\label{Ak-2}
  &\qquad =
  \sum_{m \le x_k T}
     \sum_{n \in [m / y_k, m / x_k)} \omega(n, m)
  +
  \sum_{x_k T / y_k \le n \le T}
    \sum_{m \in [x_k T, n y_k]} \omega(n, m).
\end{align}
We find it convenient to write these quantities as above in order to
keep error terms under control.

\subsection{The value of the constants \texorpdfstring{$c(b,r;i)$}{c}}

The following Lemma will take care of the constant appearing in
the main terms.

\begin{Lemma}
\label{partial-sum}
For $X \to +\infty$ we have
\[
  \sum_{b^{i-1} \le k \le X} (y_k - x_k)
  =
  \sum_{k \ge b^{i - 1}}
    \frac{b^i}{(b k + r) (b k + r + 1)}
  +
  \Odi{X^{-1}}.
\]
Furthermore
\begin{equation}
\label{constant}
  \sum_{k \ge b^{i - 1}}
    \frac{b}{(b k + r) (b k + r + 1)}
  =
  \psi\Bigl( \frac{b^i + r + 1}b \Bigr)
  -
  \psi\Bigl( \frac{b^i + r}b \Bigr).
\end{equation}
\end{Lemma}

\begin{proof}
Since $y_k - x_k \le k^{-2}$, the error term arising from extending
the sum over $k$ to all positive integers is $\ll X^{-1}$.
The results follow from the same property of the digamma function
quoted above: see \cite{olver2010LBC} formulae [5.7.6].
\end{proof}

\section{The proof of the theorem \ref{Main-theorem} for integers}

Our improvement over the previous results depends on the fact that we
manage to choose a specific threshold for $k$: above it, but below
\eqref{k-range}, we will estimate $\vert\bigcup_k \A_k(T) \vert$
trivially, while for $k$ less than it we will be able to keep a good
precision in evaluating $\vert \A_k(T) \vert$.

\subsection{Estimate of \texorpdfstring{$\U(T)$}{U/T}}

Now let us consider $\U(T)$. By \eqref{U-general} we have
\begin{align*}
  \vert \U(T; b, r; i) \vert
  &=
  \sum_{k = 0}^{b^{i - 1} - 1} \vert \A_k(T; b, r; i) \vert
  =
  \sum_{k = 0}^{b^{i - 1} - 1}
    \sum_{m \in [1, T]}
      \Bigl( \frac{m}{x_k} - \frac{m}{y_k} + \Odi{1} \Bigr) \\
  &=
  b^{i - 1}
  \sum_{m \in [1, T]}
    \frac{m}{b^i}
  +
  \Odi{T}.
\end{align*}
If $T \in \N$ we
have
\begin{equation}
\label{U-estim}
  \vert \U(T; b, r; i) \vert
  =
  \frac{T (T + 1)}{2b}
  +
  \Odi{T}
  =
  \frac{T^2}{2b}
  +
  \Odi{T}.
\end{equation}

\subsection{Estimate of \texorpdfstring{$\LL(T)$}{LL(T)}}

Using \eqref{Ak-2}, we see that the number of
lattice points in $\A_k(T; b, r; i)$ is
\begin{equation}
\label{first-est-Ak-alt}
  =
  \sum_{n \le x_k T}
    \Bigl(
      \Bigl[ \frac n{x_k} \Bigr]
      -
      \Bigl[ \frac n{y_k} \Bigr]
    \Bigr)
  +
  \sum_{x_k T / y_k \le n \le T}
    \bigl( [n y_k] - [x_k T] \bigr).
\end{equation}
The first term in \eqref{first-est-Ak-alt} is
\begin{align}
\notag
  =
  \sum_{n \le x_k T}
    \Bigl( \frac n{x_k} - \frac n{y_k} \Bigr)
  +
  \Odi{x_k T}
  &=
  \Bigl( \frac 1{x_k} - \frac1{y_k} \Bigr)
  \sum_{n \le x_k T} n
  +
  \Odi{\frac Tk}
  =
  \frac{y_k - x_k}{x_k y_k}
  \sum_{n \le x_k T} n
  +
  \Odi{\frac Tk} \\
\notag
  &=
  \frac{y_k - x_k}{2 x_k y_k}
  \Bigl( (x_k T)^2 + \Odi{x_k T} \Bigr)
  +
  \Odi{\frac Tk} \\
\label{mt-Ak}
  &=
  \frac{y_k - x_k}{2 y_k} x_k T^2
  +
  \Odi{\frac Tk}.
\end{align}
We further rewrite the last summand in \eqref{first-est-Ak-alt} as
\begin{align*}
  \sum_{x_k T / y_k \le n \le T}
    \bigl( [n y_k] - [x_k T] \bigr)
  &=
  \sum_{x_k T / y_k \le n \le T}
    \bigl( n y_k - \{ n y_k \} \bigr)
  -
  [x_k T]
  \Bigl( T - \Bigl[ \frac{x_k T}{y_k} \Bigr] \Bigr) \\
  &=
  I_1 - I_2,
\end{align*}
say.
We have
\begin{align}
\notag
  I_1
  &=
  \sum_{x_k T / y_k \le n \le T} n y_k
  +
  \Odi{1 + \Bigl( 1 - \frac{x_k}{y_k} \Bigr) T}
  =
  y_k
  \sum_{x_k T / y_k \le n \le T} n
  +
  \Odi{1 + \frac Tk} \\
\label{et-I-Ak}
  &=
  \frac12
  y_k
  \Bigl( T^2 - \frac{x_k^2}{y_k^2} T^2 + \Odi{T} \Bigr)
  +
  \Odi{1 + \frac Tk}
  =
  \frac12
  y_k
  \Bigl( 1 - \frac{x_k^2}{y_k^2} \Bigr)
  T^2
  +
  \Odi{1 + \frac Tk}.
\end{align}
We also have
\begin{align}
\notag
  I_2
  &=
  [x_k T]
  \Bigl( T - \Bigl[ \frac{x_k T}{y_k} \Bigr] \Bigr)
  =
  \bigl( x_k T + \Odi{1} \bigr)
  \Bigl( T - \frac{x_k T}{y_k} + \Bigl\{ \frac{x_k T}{y_k} \Bigr\} \Bigr) \\
\notag
  &=
  \frac{(y_k - x_k) x_k}{y_k} T^2
  +
  \Odi{\frac{y_k - x_k}{y_k} T}
  +
  \Odi{x_k T} \\
\label{et-II-Ak}
  &=
  \frac{(y_k - x_k) x_k}{y_k} T^2
  +
  \Odi{\frac Tk}\!.
\end{align}
Summing up from \eqref{first-est-Ak-alt}, \eqref{mt-Ak},
\eqref{et-I-Ak} and \eqref{et-II-Ak}, we have
\begin{align}
\notag
  \vert \A_k(T; b, r; i) \vert
  &=
  \frac{y_k - x_k}{2 y_k} x_k T^2
  +
  \frac12
  y_k
  \Bigl( 1 - \frac{x_k^2}{y_k^2} \Bigr)
  T^2
  -
  \frac{(y_k - x_k) x_k}{y_k} T^2
  +
  \Odi{\frac Tk} \\
\notag
  &=
  \frac{y_k - x_k}{2 y_k}
  \Bigl( x_k + (y_k + x_k) - 2 x_k \Bigr) T^2
  +
  \Odi{\frac Tk} \\
\label{final-est-Ak}
  &=
  \frac12 (y_k - x_k) T^2
  +
  \Odi{\frac Tk}\!.
\end{align}
We finally sum \eqref{final-est-Ak} over $b^{i - 1} \le k \le T / \beta(T)$,
obtaining
\[
  \sum_{b^{i - 1} \le k \le T / \beta(T)}
    \vert \A_k(T; b, r; i) \vert
  =
  \frac12 T^2
  \sum_{b^{i - 1} \le k \le T / \beta(T)}
    (y_k - x_k)
  +
  \Odi{\sum_{b^{i - 1} \le k \le T / \beta(T)} \frac Tk}\!.
\]
Using Lemma~\ref{partial-sum} with $X = T / \beta(T)$ we see that the
error term arising from the completion of the series is
$\ll T \beta(T)$.
The other error term contributes $\ll T \log T$.
Hence
\[
  \sum_{b^{i - 1} \le k \le T / \beta(T)}
    \vert \A_k(T; b, r; i) \vert
  =
  \frac12 T^2
  \sum_{k \ge b^{i - 1}}
    (y_k - x_k)
  +
  \Odi{T \beta(T) + T \log(T)}.
\]
We choose $\beta(T) = \log(T)$ and
recall~\eqref{contr-lower-triangle}, and the proof is complete
by~\eqref{constant} and~\eqref{U-estim}.

\section{Primes with weights: Approach via \texorpdfstring{$\theta$}{theta}}
\label{sec:primes}

In this section we prove Theorem~\ref{th:theta}.
With notation as in section~\ref{sec:basic} and splitting the sets in the
same way, we set
\[
  N_k(T)
  :=
  \sum_{(p, q) \in \A_k(T)} (\log p)(\log q).
\]
We write $(p, q)$ for prime numbers in place of $(n, m)$ and use
logarithmic weights in order to exploit the linearity of the main term
of the Chebyshev $\theta$-function.
This is critical in order to avoid the introduction of special
functions whose behaviour is hard to estimate carefully over the range
of values of $k$ that we need.
We give more details at the end of this section.

\begin{Lemma}[Prime Number Theorem]
\label{Lemma-PNT}
There exists a positive constant $c$ such that
\[
  \pi(x) = \li(x) + \Odi{x \e^{-c\sqrt{\log x}}},
\quad \text{as $x\to +\infty$.}
\]
If the Riemann Hypothesis is true, then the error term on right-hand
side may be replaced by $\Odi{x^{1/2} \log(x)}$.
\end{Lemma}

We will need repeatedly the following simple lemma, whose proof by
partial summation is straightforward.

\begin{Lemma}
\label{cor-pnt}
Let $f \colon \R^+ \to \R^+$ be a smooth increasing function, and let
$\E$ be defined by \textup{\eqref{def-E}}.
Then
\[
  \sum_{p \le X} f(p)
  =
  \int_2^X \frac{f(t)}{\log(t)} \, \dx t
  +
  \Odi{f(X) \mathcal{E}(X)}
  \qquad\text{as $X \to +\infty$.}
\]
\end{Lemma}

\subsection{The contribution from the set \texorpdfstring{$\U(T)$}{U(T)}}

Using Lemma~\ref{cor-pnt} with $f(t) = t \log(t)$ and
recalling~\eqref{def-Ak-U} and \eqref{U-general}, for
$0 \le k < b^{i-1}$ we have to evaluate
\begin{align*}
\notag
  \sum_{q \le T}
    \Bigl(
      \theta \Bigl( \frac q{x_k} \Bigr)
      -
      \theta \Bigl( \frac q{y_k} \Bigr)
    \Bigr) \log(q)
  &=
  \frac1{b^i}
  \sum_{q \le T} q \log(q)
  +
  \sum_{p \le T} \Odi{\E(p) \log(p)} \\
\notag
  &=
  \frac1{b^i}
  \int_2^T t \, \dx t
  +
  \Odi{T \E(T) \log(T)}
  +
  \Odi{\theta(T) \E(T)} \\
  &=
  \frac1{2 b^i} T^2
  +
  \Odi{T \E(T) \log(T)}.
\end{align*}
Summing over the values of $k$ mentioned above, we find that the total
contribution from the set $\U(T)$ is
\begin{equation}
\label{total-U-primes}
  \frac1{2 b} T^2
  +
  \Odi{T \E(T) \log(T)}.
\end{equation}

\subsection{The contribution from the set \texorpdfstring{$\LL(T)$}{L(T)}}

Using \eqref{Ak-2}, we see that we have to evaluate
\[
  \sum_{q \le x_k T}
    \left( \theta \Bigl(\frac q{x_k} \Bigr) \log(q)
           -
           \theta \Bigl(\frac q{y_k} \Bigr) \log(q)
    \right)
  +
  \sum_{x_k T / y_k \le p \le T}
    \bigl( \theta(p y_k) - \theta(x_k T) \bigr) \log(p)
\]
for $b^{i-1} \le k \le T / \beta(T)$.
Using Lemma~\ref{cor-pnt} with $f(t) = t \log(t)$, we deduce that the
first summand is
\begin{align*}
  \sum_{q \le x_k T}
    \left( \theta \Bigl(\frac q{x_k} \Bigr) \log(q)
           -
           \theta \Bigl(\frac q{y_k} \Bigr) \log(q)
    \right)
  &=
  \Bigl( \frac1{x_k} - \frac1{y_k} \Bigr)
  \sum_{q \le x_k T} q \log(q)
  +
  \Odi{\sum_{q \le x_k T} \E\Bigl(\frac q{x_k} \Bigr) \log(q)} \\
  &=
  \Bigl( \frac1{x_k} - \frac1{y_k} \Bigr)
  \int_2^{x_k T} t \, \dx t
  +
  \Odi{\theta(x_k T) \E(T)} \\
  &=
  \frac12 x_k^2 T^2
  \Bigl( \frac1{x_k} - \frac1{y_k} \Bigr)
  +
  \Odi{\frac Tk \E(T)}.
\end{align*}
The second summand is
\begin{align*}
  \sum_{x_k T / y_k \le p \le T}
    \bigl( \theta(p y_k) - \theta(x_k T) \bigr) \log(p)
  &=
  \sum_{x_k T / y_k \le p \le T}
    \bigl(y_k p - x_k T \bigr) \log(p) \\
  &\qquad+
  \Odi{\sum_{x_k T / y_k \le p \le T} \E(p y_k) \log(p)}.
\end{align*}
Using again Lemma~\ref{cor-pnt}, we see that the main term is
\[
  y_k \int_{x_k T / y_k}^T t \, \dx t
  -
  x_k T \int_{x_k T / y_k}^T \dx t
  +
  \Odi{y_k T \E(T)}
  =
  \frac12 y_k \Bigl( 1 - \frac{x_k}{y_k} \Bigr)^2
  T^2
  +
  \Odi{\frac Tk \E(T)}.
\]
The error term is
\begin{align*}
  \sum_{x_k T / y_k \le p \le T} \E(p y_k) \log(p)
  &\ll
  \Bigl( \pi(T) - \pi\Bigl( \frac{x_k}{y_k} T \Bigr) \Bigr)
  \E(T y_k) \log(T) \\
  &\ll
  \frac{(1 - x_k / y_k) T}{\log(T / k)} \E(T y_k) \log(T)
  \ll
  \frac Tk \E(T y_k),
\end{align*}
by the Brun-Titchmarsh inequality.

Summing up, the total contribution of the main terms is
\begin{align}
\notag
  \frac12 T^2
  \sum_{b^{i - 1} \le k \le T / \beta(T)}
    \Bigl(
      x_k^2 \Bigl( \frac1{x_k} - \frac1{y_k} \Bigr)
      +
      y_k \Bigl( 1 - \frac{x_k}{y_k} \Bigr)^2
    \Bigr)
  &=
  \frac12 T^2
  \sum_{b^{i - 1} \le k \le T / \beta(T)} (y_k - x_k) \\
\notag
  &=
  \frac{b^{i-1}}2 T^2
  \Bigl(
    \psi \Bigl( \frac{b^i + r + 1}b \Bigr)
    -
    \psi \Bigl( \frac{b^i + r}b \Bigr)
  \Bigr) \\
\label{total-L-primes-main}
  &\qquad+
  \Odi{T \beta(T)},
\end{align}
by Lemma~\ref{partial-sum}.
The total error term is
\begin{equation}
\label{total-L-primes-error}
  \ll
  T \E(T)
  \sum_{b^{i - 1} \le k \le T / \beta(T)} \frac 1k
  \ll
  T \E(T) \log(T).
\end{equation}

In order to complete the proof, we just collect
\eqref{total-U-primes}, \eqref{total-L-primes-main} and
\eqref{total-L-primes-error} and choose $\beta(T) = T^{1 / 2}$.

\subsection{Comments on the choice of weights}

We remark that using the characteristic function of the primes as the
choice of weight in this problem leads to a number of technical
complications.
If we insist on counting couples of primes \emph{without} weights,
we would find a much weaker result, with the error term smaller than
the main term just by a factor $\log(T)$.

\section{On points which are visible from the origin}

We now deal with a similar problem, where we only count points that
are visible from the origin, that is, couples $(n, m)$ with
$(n, m) = 1$.
We could use the familiar device of writing the characteristic function
of such couples by means of the M\"obius function $\mu$.
As we said in \S\ref{sec:basic}, we could write,
\[
  \omega(n, m)
  :=
  \sum_{d \mid (n, m)} \mu(d)
  =
  \begin{cases}
    1 &\text{if $(n, m) = 1$,} \\
    0 &\text{otherwise.}
  \end{cases}
\]
and then proceed along the same lines as in \S\ref{sec:basic}. This approach does not lead to any particular improvement because the extra sum on $d$ generates an additional logarithm in the error term. Nevertheless numerical data suggests a slight regularization in this case (see \S\ref{sec:plots}).

\subsection{Deduction of Theorem~\ref{th:farey} from Theorem~\ref{Main-theorem}.}

We notice that
\begin{align*}
  \vert\{ (n, m) \in [1, T]^2 \colon \phi(n / m) = r \}\vert
  &=
  \sum_{d \le T}
    \vert\{ (n, m) \in [1, T]^2 \colon
      (n, m) = d \land \phi(n / m) = r \}\vert \\
  &=
  \sum_{d \le T}
    \vert\{ (n, m) \in [1, T / d]^2 \colon
      (n, m) = 1 \land \phi(n / m) = r \}\vert.
\end{align*}
Hence, by the M\"obius inversion formula we have
\begin{align*}
  \vert\{ (n, m) \in [1, T]^2 &\colon
    (n, m) = 1 \land \phi(n / m) = r \}\vert \\
  &=
  \sum_{d \le T}
    \mu(d)
    \vert\{ (n, m) \in [1, T / d]^2 \colon \phi(n / m) = r \}\vert \\
  &=
  \sum_{d \le T}
    \mu(d)
    \Bigl(
      c(b, r; i) \frac{T^2}{d^2} + \Odi{\frac Td (1 + \log(T / d))}
    \Bigr) \\
  &=
  c(b, r; i) T^2
  \sum_{d \le T}
    \frac{\mu(d)}{d^2}
  +
  \sum_{d \le T}
    \Odi{\frac Td (1 + \log(T / d))}.
\end{align*}
The leading term is
\[
  \frac{c(b, r; i)}{\zeta(2)} T^2
  +
  \Odi{T}.
\]
The first error term contributes $\Odi{T \log(T)}$.
The second error term is $\Odi{T \log^2(T)}$.

\section{Remarks and numerical data}\label{sec:plots}

\subsection{The limit of the method}

Given an integer $b \ge 2$, pick a digit $r \in S_b$.
We choose $i = 1$ for simplicity.
We want to count the number of lattice points lying on the boundary of
the region $\A_k(T; b, r; 1)$.
We start with estimating
\begin{equation}
\label{punti-bordo}
  \Bigl\vert\Bigl\{
    (n, m) \in [1, T]^2 \colon (n, m) = 1 \land
      b \Bigl\{ \frac nm \Bigr\} = r
  \Bigr\}\Bigr\vert.
\end{equation}
We obviously have
\[
  b \Bigl\{ \frac nm \Bigr\} = r
  \qquad\Longleftrightarrow\qquad
  \Bigl\{ \frac nm \Bigr\}
  =
  \frac rb
  =
  \frac{r / (b, r)}{b / (b, r)}.
\]
The fractions at far left and far right are reduced, and this forces
\[
  m = \frac b{(b, r)}
  \qquad\text{and}\qquad
  n \equiv \frac r{(b, r)}\mod m.
\]
Let $m_1 = b / (b, r)$.
Hence the cardinality of the set in \eqref{punti-bordo} is
$\sim T / m_1$.
A similar argument shows that
\[
  \Bigl\vert\Bigl\{
    (n, m) \in [1, T]^2 \colon (n, m) = d \land
      b \Bigl\{ \frac nm \Bigr\} = r
  \Bigr\}\Bigr\vert
  \sim
  \frac T{d m_1}
\]
for $d \le T / m_1$.
Hence the total number of lattice points on the boundary of the set
$\A_k$ is
\begin{equation}\label{limit_method}
  \sim
  \frac T{m_1} \sum_{d \le T / m_1} \frac1d
  \sim
  \frac {(b, r)}b T \log(T).
\end{equation}
This seems to be the limit of our method for the proof of
Theorem~\ref{Main-theorem}.

In the next paragraph we will show that for ``small'' $ T$ actually there seems to be a bit of difference in the convergence rate and this is evident with peaks when $\frac{(b, r)}b$ is large.

\subsection{Numerical data}

We collect here some histograms, created using Wolfram Mathematica, obtained with the numerical computations that we performed for the problem of digits. Some variants were considered trying to understand if a different counting function could help to regularize the problem and reduce the error term. In every histogram, on the $x$-axis we have the digit $r$, and the height of the bar represents $\Phi(T; b, r; 1)$ (in the cases with $i>1$ the difference is less noticeable so we did not report examples) normalized by dividing by $\sum_{r=0}^{b-1}\Phi(T; b, r; 1)$. We recall that $\Phi(T; b, r; 1) := \abs{\A(T; b, r; 1)}$.
The continuous line, instead, is the graph of $c(b, r; i)$ as a function of $r$; we indicated with a dot the values corresponding to integer values of $r$, which are the ones appearing in the theorems.

\newlength{\mywidth}
\setlength{\mywidth}{9.4cm}
\newlength{\mywidthh}
\setlength{\mywidthh}{9.2cm}

We do not report cases with a very large sample ($T = 10000$) because in these circumstances  the histogram is really close to the expected value. But for smaller $T$ we can see some phenomenon taking place: for $T = 100$, the irregularities seem all but random. If you look at the case $b=30$ (see figure \ref{30_100}), which has many divisors, some numerical values noticeably exceed the expected ones.
We do not report the plots where $b$ is prime because they actually show the behavior we expect and that we have calculated in \eqref{limit_method}: an anomalous peak in $0$ and a monotonically decreasing trend.

\begin{figure}[!htbp]
\centering
\includegraphics[width=\mywidth]{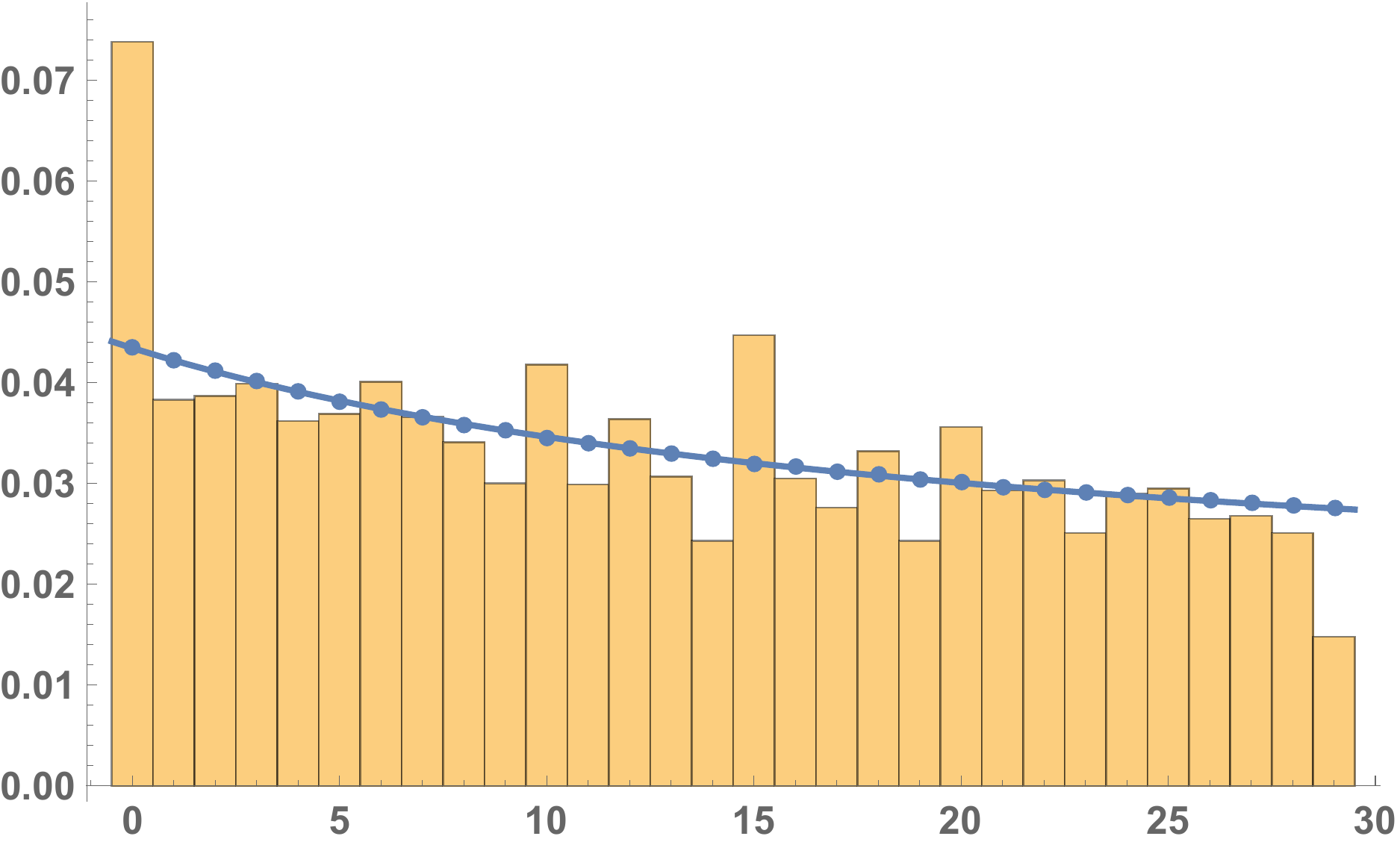}
\caption{\label{30_100}The histogram for $\abs{\A(100; 30, r; 1)}$.}
\end{figure}

Actually, the main problem for small $T$ is a simple fact of multiplicities: fractions with small numerator and denominator (like $1$, $1/2$, $1/3$, $2/3$, etc.) will be counted many times: in both of them we can recognize high counting values for the fractions that we have just mentioned. 

To avoid this phenomenon, we can just count the fractions with multiplicity one, which means taking just the reduced ones: in Figures \ref{c30_100}, we represent $\Phi(T; b, r; 1)$, with $(n,m)=1$.

 There is a regularization but some other phenomenon emerged: it seems that the divisors and in general the numbers with prime factors in common with $30$ tend to be have higher values. The explanation for this lies in the discontinuity of the system of digits: if we perturbed just a bit a number that has no digits to the right of the first one, we could reduce its first digit by one. To be more formal,  if a rational number admits a finite representation (when, after reducing the fraction, the denominator divides the base), it admits also a periodic infinite one. Just to make an example, if we think about the base $10$, we have $1 = 0.\bar{9}$.

\begin{figure}[!htbp]
\centering
\includegraphics[width=\mywidth]{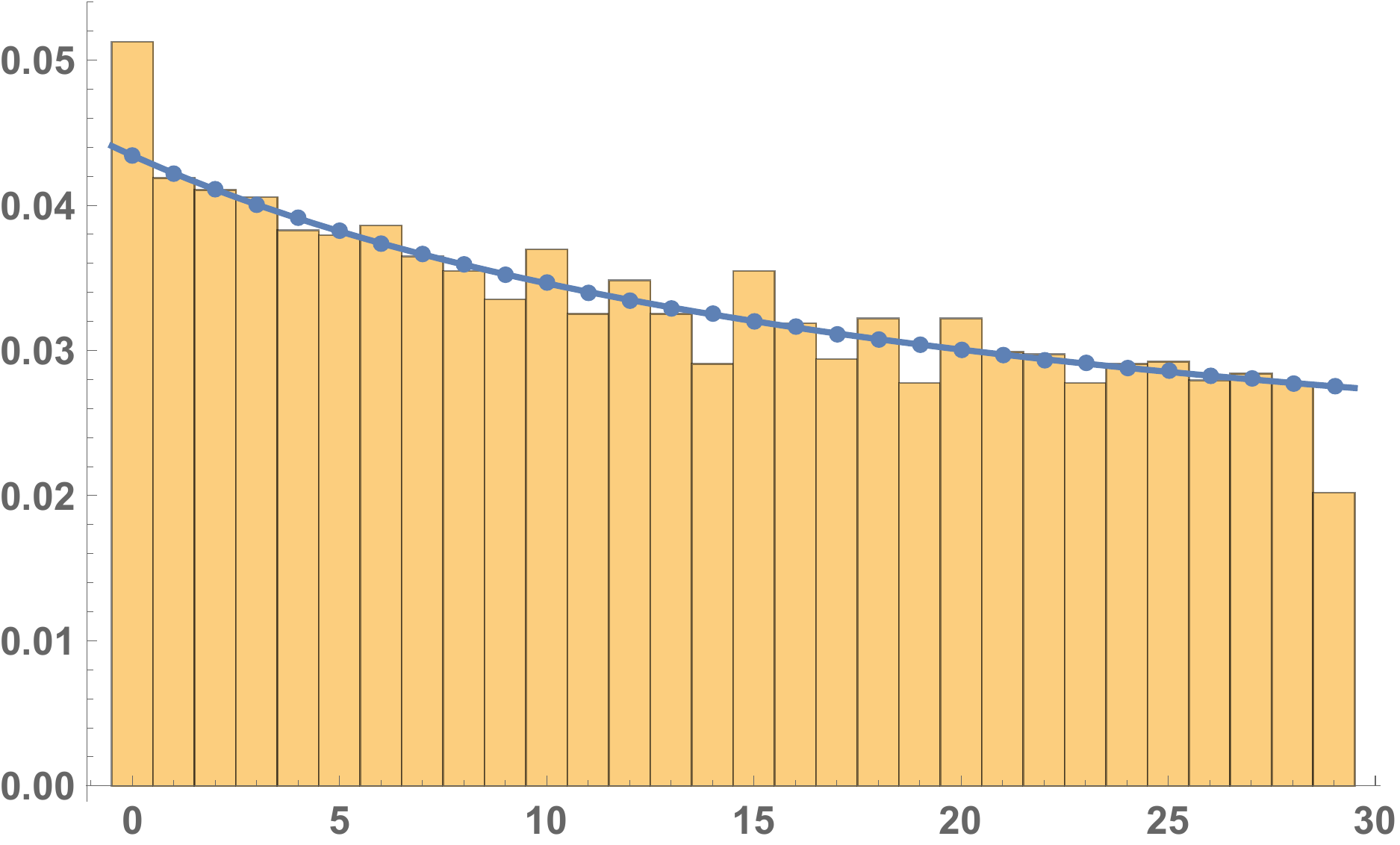}
\caption{\label{c30_100}The histogram for $\abs{\A(100; 30, r; 1)}$ (with $(n,m)=1)$.}
\end{figure}

So, in our case, any number that admits any ambiguity in its first digit in base $b$ should morally be divided into the two digits with an equal weight: assigning half weight to two digits if $b\{n/m\}\in\Z$ the rate of convergence is very good already for small values of $T$. See Figure \ref{d30_100}.

\begin{figure}[!htbp]
\centering
\includegraphics[width=\mywidthh]{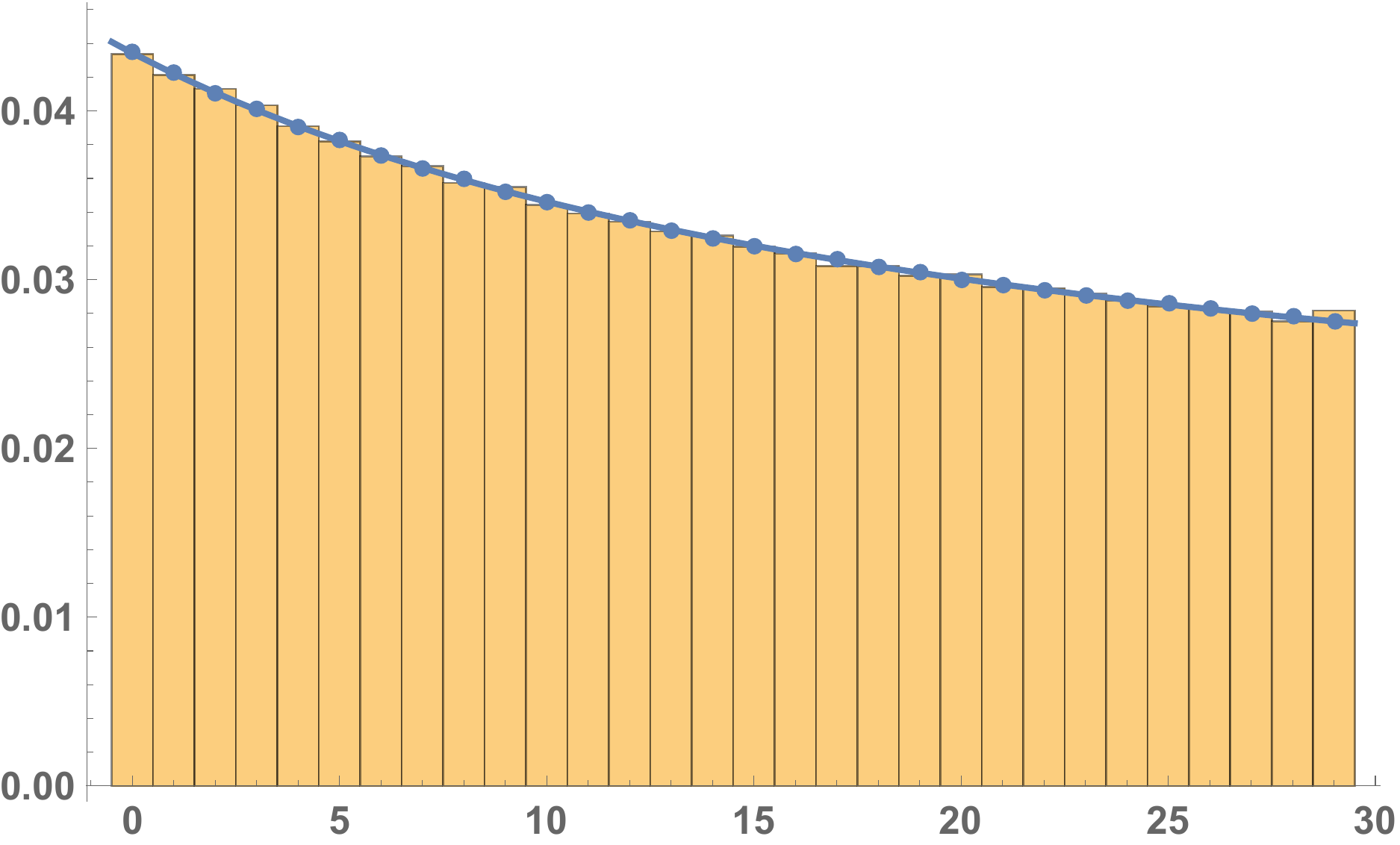}
\caption{\label{d30_100}The histogram for $\A(100; 30, r; 1)$ with $(n,m)=1$ and assigning half weight to two digits if $b\{n/m\}\in\Z$.}
\end{figure}

\subsection{Primes}
We made similar computations also in the case of prime numbers. Two positive primes are not coprime if and only if they are equal: to avoid such a possibility, we just ignore the diagonal $p=q$ in every histogram. 

Figure \ref{pd17_1000} represents the counting function that we obtained assigning half weight to two digits if $b\{p /q\}\in\Z$.

\begin{figure}[!htbp]
\centering
\includegraphics[width=\mywidthh]{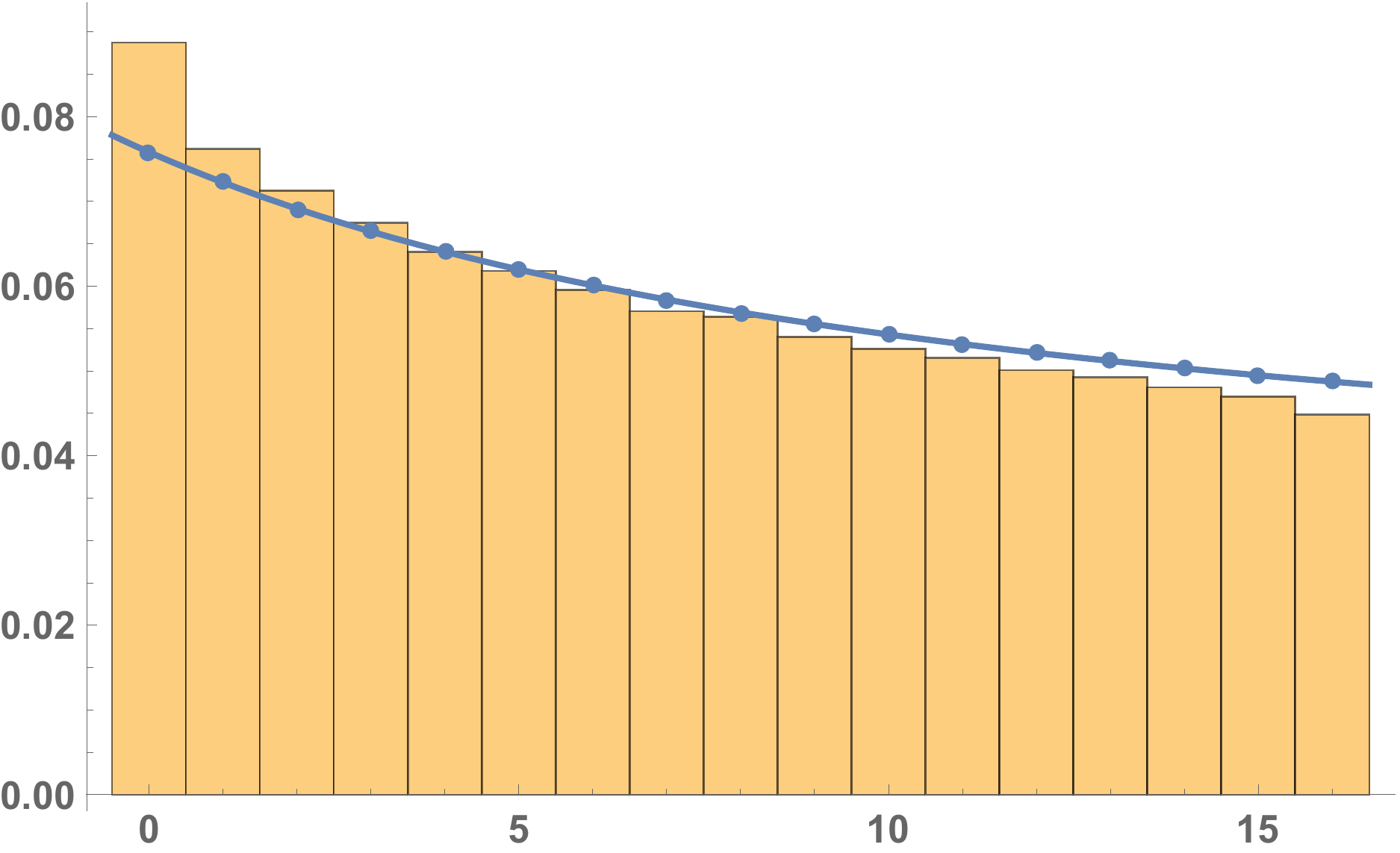}
\caption{\label{pd17_1000}The histogram for $\A(1000; 17, r; 1)$ for primes $p\neq q$, assigning half weight to two digits if $b\{p/q\}\in\Z$.}
\end{figure}

The trend is however monotonic and decreasing also for primes; the assignment of the half weight influences only the cases $r=0$ and $r=b-1$ with the effect of a further regularization. This regularity is intrinsic in the fact that we automatically exclude points with multiplicity, so we would expect something better than $T\log T$ even if we cannot prove it. The leading constant is decreasing in $r$ and it is about $0$ for $r$ around $b/2$. This would explain why, in the above figures, for $T$ sufficiently large, the first blocks are always above the trend line, while the latest are below.

\end{document}